\documentclass{aims}
\usepackage{amsmath}
  \usepackage{paralist}
  \usepackage{graphics} 
  \usepackage{epsfig} 
\usepackage{graphicx}  \usepackage{epstopdf}
 \usepackage[colorlinks=true]{hyperref}
\hypersetup{urlcolor=blue, citecolor=red}

\usepackage{comment}
\def\currentvolume{X}
 \def\currentissue{X}
  \def\currentyear{200X}
   \def\currentmonth{XX}
    \def\ppages{X--XX}
     \def\DOI{10.3934/xx.xx.xx.xx}

\newtheorem{theorem}{Theorem}[section]
\newtheorem{corollary}{Corollary}

\newtheorem{lemma}[theorem]{Lemma}

\theoremstyle{definition}
\newtheorem{definition}[theorem]{Definition}
\newtheorem{remark}{Remark}

\title[Steplength Thresholds for Invariance Preserving] 
      {Steplength Thresholds for Invariance Preserving of Discretization Methods of Dynamical Systems on a Polyhedron}

\author[Zolt\'{a}n Horv\'{a}th, Yunfei Song and and Tam\'{a}s
Terlaky]{}


 \email{horvathz@sze.hu}
 \email{yus210@lehigh.edu}
 \email{terlaky@lehigh.edu}

\thanks{ }

\begin{document}
\maketitle

\centerline{\scshape Zolt\'{a}n Horv\'{a}th }
\medskip
{\footnotesize
 \centerline{Department of Mathematics and Computational Sciences}
\centerline{Sz\'{e}chenyi Istv\'{a}n
University}
   \centerline{ 9026 Gy\H{o}r, Egyetem t\'{e}r 1, Hungary}
} 

\medskip

\centerline{\scshape Yunfei Song and Tam\'{a}s
Terlaky}
\medskip
{\footnotesize
 \centerline{Department of Industrial and Systems Engineering}
   \centerline{Lehigh University}
   \centerline{200 West Packer Avenue, Bethlehem, PA,
18015-1582, USA}
}

\bigskip

 \centerline{(Communicated by  Kok Lay Teo)}

\begin{abstract}
Steplength thresholds for invariance preserving of three types of  discretization methods on a polyhedron are considered. For Taylor approximation type discretization methods we prove that a valid steplength threshold can be obtained  by finding the first positive zeros of a finite number of polynomial functions. Further, a simple and efficient algorithm is proposed to numerically compute the steplength threshold.  For rational function type discretization methods we derive a valid  steplength threshold for invariance preserving, which can be computed by using an analogous algorithm as in the first case. The relationship between the previous two types of discretization methods and the forward Euler method is studied. Finally, we show that, for the forward Euler method,  the largest steplength threshold for invariance preserving can be computed by solving a finite number of linear optimization problems. 
\end{abstract}

\section{Introduction}

Invariant set is an important concept in the theory of dynamical systems and it has a wide range of applications in control. One of the reasons of the interest is due to the fact that invariant sets enable us to estimate the attraction region of a dynamical system. 
We consider linear continuous dynamical systems in the form 
\begin{equation}\label{dyna1}
\dot{x}(t)=A_cx(t),
\end{equation}
and discrete dynamical systems in the form 
\begin{equation}\label{dyna2}
x_{k+1}=A_dx_k,
\end{equation}
where $x_k, x(t)\in \mathbb{R}^n$ are the state variables, $A_c,A_d\in \mathbb{R}^{n\times n}$ are the coefficient matrices, and $t\in
\mathbb{R}$ and $k\in \mathbb{N}$ indicate continuous and discrete time steps, respectively. Note that equations (\ref{dyna1}) and (\ref{dyna2}) can be treated as autonomous systems or as controlled systems. In the latter case, the coefficient matrix $A_c$ $ (\text{or }A_d)$ can be represented in the form of $A+BF$, where $A$ is the open-loop state matrix, $B$ is the control matrix, and $F$ is the gain matrix. For simplicity, we use the term system to indicate dynamical system. 

Intuitively, a set $\mathcal{S}$ is called an invariant set for a  system, if all the trajectories of the  system, which are starting in $\mathcal{S}$, remain in $\mathcal{S}$. 
Numerous surveys on the theory and applications of invariant sets are published in the recent decades, see e.g.,  Blanchini \cite{Blanchini}. Recently, several sufficient and necessary conditions, which are simply refereed to as invariance conditions, are derived to verify if a set is an invariant set for a continuous or discrete system. Various convex sets with different characteristics are considered as candidates for invariant sets.  Invariance conditions for polyhedra are given in \cite{Blanchini3, blan5, caste, dorea2, dorea}. Ellipsoidal sets as invariant sets are analyzed in \cite{boyd}. Cones as invariant sets are studied in \cite{loewy, stern, Vander}. A novel unified approach to derive invariance conditions for polyhedra, ellipsoids, and cones is presented in \cite{song1}.

Although many mathematical techniques are developed to directly solve continuous  systems,  in practice, one usually solves a continuous  system by applying certain discretization methods.  Assume that a set is an invariant set for a continuous  system, then it should be also an invariant set for the discrete  system, which is obtained by the discretization method, i.e., discretization should preserve the invariance. However, this is not always true for every steplength used in the discretization method, thus it will be convenient if there exists a predictable threshold for valid invariance preserving steplength. The existence of such steplength thresholds of invariance preserving on various sets is thoroughly studied in \cite{song2}. In this paper, we consider three types of discretization methods on polyhedra and we aim to derive valid thresholds of the steplength in terms of explicit form or obtained by using efficiently computable algorithms. The popularity of polyhedra as invariant sets is  due to  the fact that the state and control variables are usually represented  in terms of linear inequalities. For Taylor approximation type discretization methods, i.e., the coefficient matrix of the discrete  system is derived from the Taylor expansion of $e^{A_c\Delta t}$, we present an algorithm to derive a valid steplength threshold for invariance preserving. In particular, the algorithm  aims to find the first positive zeros of some polynomial functions related to the  system and the polyhedron.  For  general rational function type discretization methods, i.e.,  the coefficient matrix of the discrete  system is a rational function with respect to $A_c$ and $\Delta t$,  we derive a valid steplength threshold  for invariance preserving that can be computed by using analogous methods as for the case of Taylor approximation type methods. This steplength threshold is related to the steplength threshold for the forward Euler method and the radius of absolute monotonicity of the discretization method. We note that this result is similar to the one presented in \cite{horv98, horv05}, where Runge-Kutta methods are considered.  Finally, we propose an optimization model to find the largest steplength threshold for the forward Euler method. We note that some results on the use of the forward Euler method to analyze invariance for  continuous dynamical systems can be found in \cite{Blanchini3, blan5}. 

\emph{Notation}: For the sake of simplicity, the following notational conventions are introduced.
 A \emph{nonnegative} matrix, denoted by $H\geq0$, means that all entries
of $H$ are nonnegative. An \emph{off-diagonal nonnegative} matrix,
denoted by $H\geq_o0$, means that all entries,
except the diagonal entries, of $H$ are nonnegative.

The paper is organized as follows. In Section \ref{sec:back}, some fundamental concepts, theorems, and the key problems in this paper are introduced. In Section \ref{sec:main}, we present our main results, i.e., deriving valid steplength thresholds for invariance preserving,  for the three types of discretization methods. Finally, conclusions are provided in Section \ref{sec:con}.

\section{Background}\label{sec:back}

We now introduce the definitions of invariant sets for continuous and discrete systems. 
\begin{definition}\label{def1}
A set $\mathcal{S}$ in $\mathbb{R}^n$ is an invariant set for
\begin{itemize}
\item the continuous system (\ref{dyna1}) if $x(0)\in $ $\mathcal{S}$ implies
$x(t)\in \mathcal{S}$,  for all $t\geq0$.
\item
the discrete system (\ref{dyna2}) if  $ x_k\in \mathcal{S}$
implies $ x_{k+1}\in \mathcal{S}$, for all
 $k\in \mathbb{N}$.
\end{itemize}
\end{definition}

According to the definitions of invariant sets, we have that an invariant set means that the continuous (or discrete) trajectory of the system remains in the same set. In fact, there is an alternative perspective, see e.g., \cite{song1}. In that interpretation  $\mathcal{S}$ is an invariant set for (\ref{dyna1}) if and only if $e^{A_ct}\mathcal{S}\subseteq\mathcal{S}$ for any $t\geq0$, and $\mathcal{S}$ is an invariant set for (\ref{dyna2}) if and only if $A_d\mathcal{S}\subseteq\mathcal{S}.$  

In this paper, candidate invariant sets are restricted to convex polyhedron in $\mathbb{R}^n$.  A  polyhedron $\mathcal{P}$ in $ \mathbb{R}^n$ can be characterized as the intersection of a finite number of half spaces. 
\begin{definition}
A  polyhedron $\mathcal{P}$ in $ \mathbb{R}^n$ is defined as 
\begin{equation}\label{poly1}
\mathcal{P}=\{x\in \mathbb{R}^n\,|\,g_1^Tx\leq b_1, g_2^Tx\leq b_2,...,g_m^Tx\leq b_m\}:=\{x\in \mathbb{R}^n\,|\,Gx\leq b\},
\end{equation}
where $g_1,g_2,...,g_m\in \mathbb{R}^n,$ $b\in\mathbb{R}^m,$ and $G^T=[g_1,g_2,...,g_m]\in \mathbb{R}^{n\times m} $.
\end{definition}

Two classical subsets of polyhedra are extensively studied in many applications. One is called polytope, which is a bounded polyhedron. The other one is called polyhedral cone, a polyhedron with $b=0$ in (\ref{poly1}), and the origin is its only vertex.

Given a system and a polyhedron, the invariance condition indicates sufficient and necessary condition such that the polyhedron is an invariant set for the system. There are many such equivalent invariance conditions, e.g., \cite{bits1, caste}. The most common ones are presented in Theorem \ref{poly}.  A novel and unified approach to derive these invariance conditions is proposed in \cite{song1}. The invariance conditions in Theorem \ref{poly} provide powerful and practical tools to verify whether a polyhedron is an invariant set for a given system. 

\begin{theorem}\label{poly}
\emph{\cite{bits1, caste, song1}}
A polyhedron $\mathcal{P}$ given in the form of (\ref{poly1}) is an invariant set for
\begin{itemize}
\item the continuous system (\ref{dyna1}) if and only if there exists an ${H}\in \mathbb{R}^{m\times m}$, such that 
\begin{equation}\label{invcont}
H\geq_o0, ~ HG=GA_c, \text{ and } Hb\leq0.
\end{equation}
\item the discrete system (\ref{dyna2}) if and only if there exists  an $\tilde{H}\in \mathbb{R}^{m\times m}$, such that
\begin{equation}\label{invdis}
\tilde{H}\geq0, ~\tilde{H}G=GA_d, \text{ and } \tilde{H}b\leq b.
\end{equation} 
\end{itemize}
\end{theorem}

From the theoretical perspective, when a discretization method is applied to a continuous system, the invariant polyhedron for the continuous system should also be an invariant set for the discrete system. This means that conditions (\ref{invcont}) and (\ref{invdis}) are satisfied simultaneously, when the system, polyhedron, and discretization method are given.  However, this is not always true. Intuitively, the smaller steplength used in the discretization method has larger possibility to yield that the polyhedron is also an invariant set for the discrete system. For the sake of self-contained presentation, the formal definitions of invariance preserving and steplength threshold are introduced as follows.

\begin{definition}\label{invpre}
Assume a polyhedron $\mathcal{P}$ is an invariant set for the continuous system (\ref{dyna1}), and a discretization method is applied to the continuous system to yield a discrete system.  If there exists a $\tau>0,$ such that $\mathcal{P}$ is also an invariant set for the discrete system for any steplength $\Delta t\in[0,\tau]$, then the discretization method is \textbf{invariance preserving} for $\Delta t\in[0,\tau]$ on $\mathcal{P},$ and $\tau$ is a \textbf{steplength threshold} for invariance preserving of this discretization method on $\mathcal{P}.$
\end{definition}

The steplength threshold in Definition \ref{invpre} implies that any value smaller than this threshold is also a valid steplength threshold\footnote{This is a key reason why the problem of finding a valid steplength threshold is not an easy problem. In the interval $[0,\tau]$, one needs to check every $\Delta t$ in this interval, which means that there are infinitely  many values  to be considered.}. This is an important property. In certain cases, a discretization method may be invariance preserving on a set in the form of $[0,\tau_1]\cup[\tau_2,\tau_3]$, where $\tau_1<\tau_2.$  Here we are only interested in finding $\tau_1$. We also note that the steplength threshold in Definition \ref{invpre} is uniform\footnote{This is another key reason why the problem of finding a valid steplength threshold is not an easy problem.} on $\mathcal{P}$, i.e., $\tau$ needs to be a valid steplenth threshold for every initial  point in $\mathcal{P}$.

Since a continuous system is usually solved by using various discretization methods in practice,  invariance preserving property of the chosen discretization method plays an important role. Further,  a larger steplength threshold has many advantages in practice. For example, for larger steplength, the size of the discretized system is smaller, which yields that the computation is less expensive.  Thus, we introduce the key problem in the paper:

\begin{quote}
\emph{Find a valid (if possible the largest) steplength threshold $\tau>0,$ such that a discretization method is invariance preserving for every $\Delta t\in [0,\tau]$ on $\mathcal{P}$.}
\end{quote}

\section{Main Results}\label{sec:main}
In this section, we present the approaches for computing a valid (or largest) steplength threshold such that three classes of discretization methods are invariance preserving on a polyhedron.  These three  classes of discretization methods are considered in the following order: Taylor approximation type discretization methods, rational function type discretizatin methods, and the forward Euler method. The Taylor approximation type represents a family of explicit methods. The rational function type is an extended family of the Taylor approximation type, which also includes some implicit methods. The relationship between these discretization methods and the forward Euler method is also studied. Finally, for the forward Euler method, we derive the largest steplength threshold for invariance preserving. 

\subsection{Taylor Approximation Type Discretization Methods}

We first consider the Taylor approximation type discretization methods. Note that the solution of the continuous system (\ref{dyna1}) is explicitly represented as $x(t)=e^{A_ct}x_0$, thus one can use the Taylor approximation  to numerically solve the continuous system.  The $p$-order Taylor approximation of $e^{A_c\Delta t}$ is given as follows:
\begin{equation}\label{eq13}
e^{A_c\Delta t}\approx I+A_c\Delta t+\frac{1}{2!}A_c^2\Delta t^2+\cdots+\frac{1}{p!}A_c^p\Delta t^p=\sum_{i=0}^p \frac{1}{i!}A_c^i\Delta t^i:=A_d.
\end{equation}
The discrete system obtained by applying the Taylor approximation type discretization methods is given as
$x_{k+1}=A_dx_{k},$ where $A_d$ is defined by (\ref{eq13}). In fact, the Taylor approximation type methods form a family of discretization methods. 
For example,  $p=1$ corresponds to the forward Euler method,  $p=2$ corresponds to the general Runge-Kutta 2nd order methods. 

\subsubsection{Existence of Steplength Threshold}
Our approach to derive  steplength threshold is based on the invariance conditions presented in 
Theorem \ref{poly}. The basic ideas is that we build the relationship between these two invariance conditions of the continuous and discrete systems. In fact, conditions (\ref{invcont}) and (\ref{invdis}) are essentially linear feasibility problems \cite{roos}. 
The unknowns in the two invariance conditions are the matrix $H$ and $\tilde{H}$ given by (\ref{invcont}) and (\ref{invdis}), respectively. Thus, the key is to find  relationship between those matrices. 

\begin{lemma}\label{lemgam}
\emph{\cite{horn}} Assume $H$ satisfies (\ref{invcont}), then there exists $\gamma>0$, such that $
\hat{H}=H+\gamma I\geq0.$ 
\end{lemma}
\begin{proof}
Since $H\geq_o0,$ we can choose $\gamma>\max\{0,-\min\{h_{ii}, ~1\leq i\leq n\}\}$, which yields 
$H+\gamma I\geq0.$ The result is immediate by taking $\hat{H}=H+\gamma I,$ 
\end{proof}

We note that $\gamma$ in Lemma \ref{lemgam} is not unique, e.g., any value greater than a valid $\gamma$ is also valid.  We will show more about the effect of $\gamma$ to the steplength threshold in Section \ref{sec:ration}, and the way to derive a larger steplength threshold based on $\gamma$ is also presented. 

\begin{lemma}\label{leminvar}
Assume $H$ satisfies (\ref{invcont}), and define 
\begin{equation}\label{p2:eq3}
\tilde{H}(\Delta t)=I+H\Delta t+\frac{1}{2!}H^2\Delta t^2+\cdots+\frac{1}{p!}H^p\Delta t^p=\sum_{i=0}^p \frac{1}{i!}H^i\Delta t^i.
\end{equation}
\begin{itemize}
\item[a).] For the $\gamma$ and $\hat{H}$ given in Lemma \ref{lemgam}, we have 
\begin{equation}\label{lemma453:eq1}
\tilde{H}(\Delta t)=f_0(\Delta t) I+f_1(\Delta t) \hat{H}+...+f_p(\Delta t) \hat{H}^p,
\end{equation}
where
\begin{equation}\label{lemma453:eq3}
f_i(\Delta t)=\sum_{k=i}^p\frac{(-1)^{k-i}}{k!}\binom{k}{i}\gamma^{k-i}\Delta t^k, \text{ for } i=0,1,...,p,
\end{equation}
and 
\begin{equation}\label{lemma87691}
\sum_{i=0}^p \gamma^if_i(\Delta t)=1.
\end{equation}
\item [b).]  Let $\tau=\min_{i=0,...,p}\{\tau_i\},$ where $\tau_i$ is the first positive zero of $f_i(\Delta t).$ Then for all $\Delta t\in [0,\tau],$ the matrix $\tilde{H}(\Delta t)$ satisfies (\ref{invdis}),
where $A_d$ is defined by (\ref{eq13}).
\end{itemize}
\end{lemma}

\begin{proof}
a). According  to Lemma \ref{lemgam},  there exists $\gamma>0$, such that $
\hat{H}=H+\gamma I\geq0.$  
The matrix $\tilde{H}(\Delta t)$ given by (\ref{p2:eq3}) is represented in terms of $\Delta t$. By substituting $H=\hat{H}-\gamma I$ into (\ref{p2:eq3}), we now reformulate $\tilde{H}(\Delta t)$ in terms of $\hat{H}$, i.e.,
\begin{equation}\label{lem45:eq2}
\begin{split}
\tilde{H}(\Delta t)&=I+(\hat{H}-\gamma I)\Delta t+\frac{1}{2!}(\hat{H}^2-2\gamma \hat{H}+\gamma^2I)\Delta t^2+\cdots\\
&~~~~~~+\frac{1}{p!}(\hat{H}^p-p\gamma\hat{H}^{p-1}+\cdots+(-1)^p \gamma^p I)\Delta t^p.\\
\end{split}
\end{equation}
According to (\ref{lem45:eq2}), the coefficients of $\hat{H}^i$, for $i=0,1,...,p$,  is given as 
\begin{equation*}
\frac{1}{i!}\Delta t^i+\frac{-1}{(i+1)!}\binom{i+1}{i}\gamma\Delta t^{i+1}+\frac{(-1)^2}{(i+2)!}\binom{i+2}{i}\gamma^2\Delta t^{i+2}+\cdots+\frac{(-1)^{p-i}}{p!}\binom{p}{i}\gamma^{p-i}\Delta t^{p},
\end{equation*}
which is the same as (\ref{lemma453:eq3}).

We note that $\sum_{i=0}^p \gamma^if_i(\Delta t)$ is equivalent to replacing $I$ and $\hat{H}$ in (\ref{lemma453:eq1}) by 1 and $\gamma$, respectively.  Then, according to (\ref{lem45:eq2}), we have
\begin{equation}\label{lemma453:eq31}
\sum_{i=0}^p \gamma^if_i(\Delta t)=\sum_{i=0}^p\frac{1}{i!}(\gamma \Delta t)^i\sum_{k=0}^i(-1)^k\binom{i}{k}.
\end{equation}
For  $i>0$, we have $\sum_{k=0}^i(-1)^k\binom{i}{k}=(x-1)^{i}|_{x=1}=0,$ implying that the right hand side of (\ref{lemma453:eq31}) equals to 1, thus (\ref{lemma87691}) follows immediately.

b).  We note that for every $i$ the first term of $f_i(\Delta t)$ given as in (\ref{lemma453:eq3}) is $\frac{1}{i!}\Delta t^i$.  Then we can write 
\begin{equation}\label{func}
f_i(\Delta t)=\frac{\Delta t^i}{i!}\big(1+\mathcal{O}(\Delta t)\big).
\end{equation}
Thus, we have that there exists a $\tau_i>0$, i.e., the first positive zero of $f_i(\Delta t),$ where $\tau_i$ may be infinity, such that $f_i(\Delta t)\geq0$ for all $\Delta t \in[0,\tau_i].$ Then we let 
\begin{equation}\label{valtau}
\tau=\min_{i=0,1,...,p}\{\tau_i\},
\end{equation}
thus we have $f_i(\Delta t)\geq0$ for all $\Delta t \in[0,\tau]$ and $i=0,1,...,p.$ According to (\ref{lemma453:eq1}), and by noting that $\hat{H}^i\geq0$ for any $i=1,2,...,p$, we have that $\tilde{H}(\Delta t)\geq0$ for all $\Delta t\in [0,\tau],$ where $\tau$ is defined by (\ref{valtau}).  Thus, we have proved that the first condition in (\ref{invdis}) is satisfied.

By recursively using $HG=GA_c$, for any $i,$ we have 
\begin{equation}\label{p21:eq1}
H^iG=H^{i-1}(HG)=H^{i-1}GA_c=H^{i-2}(HG)A_c=H^{i-2}GA_c^2=...=GA_c^i.
\end{equation}
Then, according to  (\ref{p21:eq1}), and substituting (\ref{p2:eq3}) and (\ref{eq13}), we have
\begin{equation*}
\tilde{H}(\Delta t)G=\sum_{i=0}^p \frac{1}{i!}H^iG\Delta t^i=\sum_{i=0}^p \frac{1}{i!}GA^i\Delta t^i=G\sum_{i=0}^p \frac{1}{i!}A^i\Delta t^i=GA_d.
\end{equation*}
Thus, we have proved that the second condition in (\ref{invdis}) is satisfied. 

Since $H$ satisfies (\ref{invcont}), we have $Hb\leq0$. Also, note that $H=\hat{H}-\gamma I$, thus we have
$
(\hat{H}-\gamma I)b\leq0, \text{ i.e., } \frac{\hat{H}}{\gamma}b\leq b.
$
Since $\frac{\hat{H}}{\gamma}\geq0$, we have
\begin{equation}\label{eq3875o1}
\Big(\frac{\hat{H}}{\gamma}\Big)^ib\leq b, \text{ i.e., } \hat{H}^ib\leq \gamma^ib, \text{ for any } i=1,2,...,p.
\end{equation}
Then, according to (\ref{eq3875o1}) and (\ref{lemma87691}), we have
\begin{equation*}
\begin{split}
\tilde{H}_{\Delta t}b&=(f_0(\Delta t) I+f_1(\Delta t) \hat{H}+\cdots+f_p(\Delta t) \hat{H}^p)b\\
          &\leq (f_0(\Delta t)+\gamma f_1(\Delta t)+\cdots+\gamma^p f_p(\Delta t))b\\
          &\leq b.
\end{split}
\end{equation*}
Thus, we have proved that the third condition in (\ref{invdis}) is satisfied. The proof is complete. 
\end{proof}

Lemma \ref{leminvar} presents an important relationship between the two matrices $H$ and $\tilde{H}$ corresponding to the continuous and discrete systems invariance conditions. This relationship is explicitly represented in (\ref{p2:eq3}), which is derived from the Taylor approximation  (\ref{eq13}).  
According to Lemma \ref{leminvar} and Theorem \ref{poly}, we have the following theorem.

\begin{theorem}\label{thm123}
Assume a polyhedron $\mathcal{P}$ be given as in (\ref{poly1}) is an invariant set for the continuous system (\ref{dyna1}), and a Taylor approximation type discretization method (\ref{eq13}) is applied to the continuous system (\ref{dyna1}). Then, the  steplength threshold $\tau>0$ as given in  Lemma \ref{leminvar} is a valid steplength threshold for invariance preserving for the given Taylor approximation type discretization method (\ref{eq13}) on $\mathcal{P}$.
\end{theorem}

According to the proof of Lemma \ref{leminvar}, we have that a valid $\tau$ requires $f_i(\Delta t)\geq0$ for all $\Delta t\in[0,\tau]$ and all $i=0,1,...,p$, where $f_i(\Delta t)$ given as (\ref{lemma453:eq3}).  Since each $f_i(\Delta t)$ can be represented in the form of (\ref{func}),  the following corollary is immediate.

\begin{corollary}
The value of $\tau$ given in Theorem \ref{thm123} (or Lemma \ref{leminvar}) is a valid steplength threshold for invariance preserving on $\mathcal{P}$ for the Taylor approximation type discretization methods  (\ref{eq13}). To compute $\tau$, one needs to find the first positive zeros of finitely many  polynomial functions in the form
\begin{equation}\label{thm311:eq1}
f(\Delta t)=1+\alpha_1\Delta t+\alpha_2\Delta t^2+...+\alpha_q\Delta t^q, ~~~\alpha_q\neq0,
\end{equation} 
where $\alpha_1,\alpha_2,...,\alpha_q\in \mathbb{R}$ and $q\in \mathbb{N}$.
\end{corollary}

In fact, Lemma \ref{leminvar} can be extended to a more general case for polynomial approximation rather than Taylor type discretization methods. 
\begin{theorem}\label{leminvar1}
Assume $H$ satisfies (\ref{invcont}), and define 
\begin{equation}\label{p2:eqd3}
\tilde{H}(\Delta t)=I+\sigma_1H\Delta t+\sigma_2H^2\Delta t^2+\cdots+\sigma_pH^p\Delta t^p=\sum_{i=0}^p \sigma_iH^i\Delta t^i.
\end{equation}
\begin{itemize}
\item[a).] For the $\gamma$ and $\hat{H}$ given in Lemma \ref{lemgam}, we have 
\begin{equation}\label{lemma4d:eq1}
\tilde{H}(\Delta t)=f_0(\Delta t) I+f_1(\Delta t) \hat{H}+...+f_p(\Delta t) \hat{H}^p,
\end{equation}
where
\begin{equation}\label{lemma4d:eq3}
f_i(\Delta t)=\sum_{k=i}^p{(-1)^{k-i}}\sigma_k\binom{k}{i}\gamma^{k-i}\Delta t^k, \text{ for } i=0,1,...,p,
\end{equation}
and 
\begin{equation}\label{lemma87dd}
\sum_{i=0}^p \gamma^if_i(\Delta t)=1.
\end{equation}
\item [b).]  Let $\tau=\min_{i=0,...,p}\{\tau_i\},$ where $\tau_i$ is the first positive zero of $f_i(\Delta t).$ Then for all $\Delta t\in [0,\tau],$ the matrix $\tilde{H}(\Delta t)$ satisfies (\ref{invdis}),
where $A_d$ is defined by (\ref{eq13}).
\end{itemize}
\end{theorem}
\begin{proof}
a). According  to Lemma \ref{lemgam},  there exists a $\gamma>0$, such that $
\hat{H}=H+\gamma I\geq0.$
The matrix $\tilde{H}(\Delta t)$ given by (\ref{p2:eqd3}) is represented in terms of $\Delta t$. By substituting $H=\hat{H}-\gamma I$ into (\ref{p2:eqd3}), we now reformulate $\tilde H(\Delta t)$ in terms of $\hat{H}$, i.e.,
\begin{equation}\label{lem45:eqd2}
\begin{split}
\tilde{H}(\Delta t)&=I+\sigma_1(\hat{H}-\gamma I)\Delta t+\sigma_2(\hat{H}^2-2\gamma \hat{H}+\gamma^2I)\Delta t^2+\cdots\\
&~~~~~~+\sigma_p(\hat{H}^p-p\gamma\hat{H}^{p-1}+\cdots+(-1)^p \gamma^p I)\Delta t^p.\\
\end{split}
\end{equation}
According to (\ref{lem45:eqd2}), the coefficient of $\hat{H}^i$, for $i=0,1,...,p$,  is given as
\begin{equation*}
\sigma_i\Delta t^i-\sigma_{i+1}\binom{i+1}{i}\gamma\Delta t^{i+1}+\sigma_{i+2}\binom{i+2}{i}\gamma^2\Delta t^{i+2}+\cdots+(-1)^{p-i}\sigma_p\binom{p}{i}\gamma^{p-i}\Delta t^{p},
\end{equation*}
which is the same as (\ref{lemma4d:eq3}).

We note that $\sum_{i=0}^p \gamma^if_i(\Delta t)$ is equivalent to replacing $I$ and $\hat{H}$ by 1 and $\gamma$, respectively, in (\ref{lemma4d:eq1}).  Then, according to (\ref{lem45:eqd2}), we have
\begin{equation}\label{lemma453:eqd31}
\sum_{i=0}^p \gamma^if_i(\Delta t)=\sum_{i=0}^p\alpha_i(\gamma \Delta t)^i\sum_{k=0}^i(-1)^k\binom{i}{k}.
\end{equation}
For  $i>0$, we have $\sum_{k=0}^i(-1)^k\binom{i}{k}=(x-1)^{i}|_{x=1}=0,$ implying that the right hand side of (\ref{lemma453:eqd31}) equals to 1, thus (\ref{lemma87dd}) follows immediately.

The proof for Part b) is the same as the one presented for Part b) in Lemma \ref{leminvar}, thus we are not presenting here. 
\end{proof}

\subsubsection{Compute Steplength Threshold}\label{sec:ty2}
We now consider the value of $\tau$, i.e., the steplength threshold. In this section, we present an algorithm to numerically compute $\tau$. In particular, this algorithm aims to find the first positive zero of a polynomial function in the form of (\ref{thm311:eq1}). 

\begin{lemma}\label{lemm321}
Let $f(\Delta t)$ be given as in (\ref{thm311:eq1}). There exists a $\tau^*>0$, such that $f(\Delta t)\geq0$ for all $\Delta t\in [0,\tau^*]. $
\end{lemma}
\begin{proof}
Since $f(0)=1>0$, and $f(\Delta t)$ is a continuous function, the lemma is immediate.
\end{proof}

Let $f(\Delta t)$ be given as  in (\ref{thm311:eq1}). If $\alpha_1, \alpha_2, ...,\alpha_q\geq0$, then $f(\Delta t)\geq0$ for all $\Delta t\geq0$, which implies $\tau^*=\infty$ in Lemma \ref{lemm321}. Also, since $f(\Delta t)$ is dominated by $\alpha_q\Delta t^q$ for $\Delta t\gg1$,  we have that $\tau^*=\infty$ implies $\alpha_q>0.$ Therefore, the largest $\tau^*$ that satisfies Lemma \ref{lemm321} is the first positive zero of $f(\Delta t),$ otherwise, we have $\tau^*=\infty.$
In fact, we can find a predicted large $t^*>0$, such that if there is no zeros of $f(\Delta t)$ in $[0,t^*],$ then we have $\tau^*=\infty.$ Note that this case only occurs when $\alpha_q\Delta t^q$ dominates $f(\Delta t)$. This is presented in the following lemma. 

\begin{lemma}\label{lema:bnd}
Let $f(\Delta t)$ be given as in (\ref{thm311:eq1}) and $\alpha_q>0$. Let  $\alpha^*=\max\{1,|\alpha_1|,|\alpha_2|,$ $...,|\alpha_{q-1}|\}$ and ${t}^*=\frac{\alpha^*}{\alpha_q}+1$. Then if $f(\Delta t)$ has no real zero in $[0,t^*]$, then $f(\Delta t)>0$ for all $\Delta t>0.$
\end{lemma}
\begin{proof}
Since $f(\Delta t)$ has no real zero in $[0,t^*],$ we have $f(\Delta t)>0$ on $[0,t^*]$. Thus, we only need to prove the following holds:
\begin{equation*}
\alpha_q\Delta t^q>|1+\alpha_1\Delta t+\alpha_2\Delta t^2+...+\alpha_{q-1}\Delta t^{q-1}|, \text{ for all } \Delta t\in(t^*,\infty].
\end{equation*}
Note that $t^*=\frac{\alpha^*}{\alpha_q}+1$ implies $\alpha_q=\frac{\alpha^*}{t^*-1}>\frac{\alpha^*}{\Delta t-1}$ for all $\Delta t\in(t^*,\infty]$. Then we have 
\begin{equation*}
\begin{split}
|1+\alpha_1\Delta t+\alpha_2\Delta t^2+...+\alpha_{q-1}\Delta t^{q-1}| &\leq \alpha^*(1+\Delta t+\Delta t^2+...+\Delta t^{q-1})\\
                                   &=\alpha^*\frac{\Delta t^q-1}{\Delta t-1}<\alpha_q(\Delta t^q-1)<\alpha_q\Delta t^q.
\end{split}
\end{equation*}
The proof is complete.
\end{proof}

In fact, the value ${t}^*$ given in Lemma \ref{lema:bnd} can be considered as one of the termination criteria of the algorithm to find the first positive zero of $f(\Delta t)$, where $f(\Delta t)$ is defined as (\ref{thm311:eq1}). \\

The Sturm sequence $\{s_i(t)\}$ of $f(t)$  and the Sturm Theorem presented in the following definition play a key role in our algorithm. The Sturm Theorem aims to give the number of real zeros of a univariate polynomial function in an interval by using the property of Sturm sequence on the end points of the interval. 

\begin{definition} \emph{\cite{sturmfel}}
Let $f(t)$ be a univariate polynomial function. The \textbf{Sturm sequence} $\{s_i(t)\}, i=1,2,...$, of $f(t)$ is  defined as 
\begin{equation*}
s_0(t)=f(t),~~ s_1(t)=s'(t), ~~s_i(t)=-\emph{rem}(s_{i-2}(t), s_{i-1}(t)),~~ i\geq2,
\end{equation*}
where $s'(t)$ is the derivative of $s(t)$ with respect to $t$, and $s_i(t)$ is the negative of the remainder on division of $s_{i-2}(t)$ by $s_{i-1}(t).$
\end{definition}

For the sake of simplicity, we introduce the following definition and notation, which are used in the statement of the Sturm Theorem. 
\begin{definition}
For a sequence $\{\nu_i\}$, $i=1,2,...,q,$ the \textbf{number of sign changes}, denoted by $\#\{\nu_i\}$, is the number of the times of the signs change (zeros are ignored) from $\nu_1$ to $\nu_q$. 
\end{definition}

For example, if a sequence is given as $\{\nu_i\}=\{1,0,3,-2,0,2,-1,0,-3\}$, then the signs of the sequence are $\{+,0,+,-,0,+,-,0,-\}$. By eliminating all zeros, we have $\{+,+,-,+,-,-\}$, which  has 3 sign changes, i.e., $\#\{\nu_i\}=3.$

\begin{theorem}\label{sturm}
 \emph{\cite{sturmfel}}
\textbf{(Sturm Theorem)} Let $f(t)$ be a univariate polynomial function. If $\alpha<\beta$ and $f(\alpha), f(\beta)\neq0.$ Then the number of distinct real zeros of $f(t)$ in the interval $[\alpha,\beta]$ is equal to $|\#\{s_i(\alpha)\}-\#\{s_i(\beta)\}|,$ where $\{s_i(t)\}$ is the Sturm sequence of $f(t).$ 
\end{theorem}

According to Lemma \ref{lema:bnd} and Theorem \ref{sturm}, we now propose our algorithm to numerically find the first positive zero of $f(\Delta t)$ where $f(\Delta t)$ is defined as (\ref{thm311:eq1}). Let us denote $\#f[\delta]$ the number of positive zeros of $f(\Delta t)$ at interval $[0,\delta].$ The value of $\#f[\delta]$ can be computed by Sturm Theorem \ref{sturm}. 
The basic idea in our algorithm is by using the bisection method to shrink the interval, which contains the first positive zero of $f(t)$, by 2 in each iteration. Our algorithm is presented as follows.

\begin{description}
\item[\textbf{Step 0: [Initial Inputs]}] Set $t^\circ=1$. Iterate $t^\circ=\frac{t^\circ}{2}$ until $\#f[t^\circ]=0.$ Let $t^*$ be given as in Lemma \ref{lema:bnd}.
\item[\textbf{Step 1: [Initial Setting]}] Set $t_l= t^\circ,$ $t_r= t^*$, and $\epsilon$ be the precision.
\item[\textbf{Step 2: [Termination 1]}] If $\#f[t_r]=0,$ then $\tau=\infty$. 
\item[\textbf{Step 3: [Termination 2]}]If $\#f[t_r]=1$ and $f(t_r)=0$, then $\tau=t^*.$ 

\item[\textbf{Step 4: [Bisection Method]}] Set $t_m=\frac{t_l+t_r}{2}$. 

Repeat until  $|t_l-t_r|<\epsilon$:
\begin{itemize}
\item \textbf{[Termination 3]}  If $\#f[t_m]=1$ and $f(t_m)=0$, then $\tau=t_m.$
\item \textbf{[Update $t_r$]} If $\#f[t_m]=1$ and $f(t_m)\neq 0,$ or $\#f[t_m]>1$, then set $t_r=t_m.$
\item \textbf{[Update $t_l$]} If $\#f[t_m]=0$, then set $t_l=t_m.$
\end{itemize}
End
\item[\textbf{Step 5: [Termination 4]}] If Step 4 is terminated at $|t_l-t_r|<\epsilon$, then $\tau=t_l.$\end{description}

The correctness of the termination condition in Step 2 is ensured by Lemma \ref{lema:bnd}. If neither of the termination conditions in Step 2 and 3 are satisfied, then it means that the first positive zero of $f(t)$ exists and is located in the interval $(t_l, t_r).$
The second case in Step 4  means that the first positive zero of $f(t)$ is located in the interval  $(t_l,t_m)$.
Analogously, the third case in Step 4  means that the first positive zero of $f(t)$ is located in the interval $(t_m,t_r)$.
In Step 5, we conclude that the first positive zero of $f(t)$ is located in the interval $(t_l,t_r).$ Recall that we are interested to find a value $\tau$, such that $f(t)\geq0$ for all $[0,\tau]$, thus we return $t_l$, i.e., the left end of the interval.

\begin{remark}
If all coefficients $\sigma_i\geq0$ for $i=1,2,...,p$ in (\ref{p2:eqd3}), then the algorithm is also applicable to compute a valid steplength threshold for invariance preserving for the polynomial approximation (\ref{p2:eqd3}).
\end{remark}

\subsection{Rational Function Type Discretization Methods}\label{sec:ration}
The previous discussion is mainly about a steplength threshold for invariance preserving for a Taylor approximation type discretization methods as specified in (\ref{eq13}).  In this section, we consider more general discretization methods, which are refereed to as the rational function type discretization methods. To be specific, these discretization methods applying to the continuous system yield the discrete system 
\begin{equation}\label{eq:rat}
x_{k+1}=r(A_c\Delta t)x_k,
\end{equation} 
where $r(t):\mathbb{R}\rightarrow \mathbb{R}$ is a rational function  defined as 
\begin{equation}\label{th2:eq11}
r(t)=\frac{g(t)}{h(t)}=\frac{\lambda_0+\lambda_1t+\cdots+\lambda_pt^p}{\mu_0+\mu_1t+\cdots+\mu_qt^q},
\end{equation}
where $\lambda_0,\lambda_1,...,\lambda_p\in\mathbb{R},~ \mu_0,\mu_1,...,\mu_q\in\mathbb{R}$, and  $p,q\in\mathbb{N}$. It is clear that Taylor approximation type discretization methods belong to this type.  Some implicit methods are also in this type, e.g., the backward Euler method, Lobatto methods \cite{ernst2}, etc.  
\begin{definition}\label{lem:rat}
\emph{\cite{high1}} Let $r(t)$ be given as in (\ref{th2:eq11}), and let $M$ be a matrix. Assume $h(M)$ is nonsingular, then 
\begin{equation}\label{rat:eq1}
r(M):=(h(M))^{-1}g(M)=g(M)(h(M))^{-1}.
\end{equation}
\end{definition}

\subsubsection{Existence of Steplength Threshold}
 In this  subsection, our analysis uses the so called \emph{radius of absolute monotonicity} of a function.
\begin{definition}
\emph{\cite{spik}} Let $r(t):\mathbb{R}\rightarrow \mathbb{R}$. If $\rho=\max\{\kappa\,|\,r^{(i)}(t)\geq0 \text{ for all } i=1,2,...,\text{ and } t\in [-\kappa,0]\},$ where $r^{(i)}(t)$ is the $i^{th}$ derivative of $r(t)$, then $\rho$ is called the \textbf{radius of absolute monotonicity} of $r(t)$.
\end{definition}

The radius of absolute monotonicity of a function is extensively used in the analysis of positivity, monotonicity, and contractivity of discretization methods for ordinary differential equations, see e.g., \cite{horv98,karaa1,spik}. 

\begin{theorem}\label{thm31}
Assume $r(t)$ is a rational function with $r(0)=1$. Let $\rho$ be the radius of absolute monotonicity of $r(t)$. Assume a polyhedron $\mathcal{P}$ be given as in (\ref{poly1}) is an invariant set for the continuous system (\ref{dyna1}), and the rational function type discretization method given as in (\ref{eq:rat}) is applied to the continuous system (\ref{dyna1}).   Then $\tau=\frac{\rho}{\gamma}$,  where $\gamma$ is given in Lemma \ref{lemgam}, is a valid steplength threshold  for invariance preserving of the rational function type discretization method given as in (\ref{eq:rat}) on  $\mathcal{P}$. 
\end{theorem}
\begin{proof}
The framework of this proof is similar to the one presented for Lemma \ref{leminvar}.  
Since $\mathcal{P}$ is an invariant set for the continuous system, according to Theorem \ref{poly} and  Lemma \ref{lemgam}, there exists an $H,$ and $\gamma>0$, such that
\begin{equation}\label{eq87}
H+\gamma I\geq0, ~HG=GA_c, \text{ and } Hb\leq0.
\end{equation}
Then, according to Theorem \ref{poly}, to ensure $\mathcal{P}$ is also an invariant set for the discrete system, we need to prove that there exists an $\tilde{H}(\Delta t)\in \mathbb{R}^{m\times m}$, such that 
\begin{equation}\label{eq89}
\tilde{H}(\Delta t)\geq0, ~\tilde{H}(\Delta t)G=Gr(A_c\Delta t), \text{ and } \tilde{H}(\Delta t)b\leq b.
\end{equation}
Let $\tilde{H}(\Delta t)=r(H\Delta t)$. Now we prove that  $\tilde{H}(\Delta t)$ satisfies (\ref{eq89}).  

For the first condition in (\ref{eq89}), we use the Taylor expansion of $r(t)$ at the value $-\rho$ as 
\begin{equation}\label{thm32:eq12}
r(t)=\sum_{i=0}^{\infty}\frac{r^{(i)}(-\rho)}{i!}(t+\rho)^i.
\end{equation}
By substituting $t=H\Delta t$ into (\ref{thm32:eq12}) we have
\begin{equation}\label{eq901}
\tilde{H}(\Delta t)=r(H\Delta t)=\sum_{i=0}^{\infty}\frac{r^{(i)}(-\rho)}{i!}(H\Delta t+\rho I)^i=\sum_{i=0}^{\infty}\frac{r^{(i)}(-\rho)}{i!}(\Delta t)^i\left(H+\frac{\rho}{\Delta t}I\right)^i.
\end{equation}
Since $\rho$ is the radius of absolute monotonicity of $r(t)$, we have $\frac{r^{(i)}(-\rho)}{i!}\geq0$ for all $i$. Also, according to (\ref{eq87}), and $\Delta t\leq \frac{\rho}{\gamma}$, i.e., $\frac{\rho}{\Delta t}\geq\gamma$,  so we have $H+\frac{\rho}{\Delta t} I\geq H+\gamma I\geq0$. Then we have $(H+\frac{\rho}{\Delta t} I)^i\geq0$ for all $i$. According to (\ref{eq901}), we have  $\tilde{H}(\Delta t)\geq0$ for  $\Delta t\leq \frac{\rho}{\gamma},$ thus the first condition in (\ref{eq89}) is satisfied. 

For the second condition in (\ref{eq89}), according to Definition \ref{lem:rat},  the second condition in (\ref{eq89}) can be rewritten as $(h(H\Delta t))^{-1}g(H\Delta t)G=Gg(A_c\Delta t)(h(A_c\Delta t))^{-1}$, i.e.,
\begin{equation}\label{th2:eq12}
g(H\Delta t)Gh(A_c\Delta t)=h(H\Delta t)Gg(A_c\Delta t).
\end{equation}
According to (\ref{th2:eq11}), we have 
\begin{equation}\label{th2:eq13}
\begin{split}
h(H\Delta t)Gg(A_c\Delta t)&=\sum_{i=1}^p\sum_{j=1}^q\lambda_i\mu_jH^iGH^j\Delta t^{i+j},\\
 g(H\Delta t)Gh(A_c\Delta t)&=\sum_{j=1}^q\sum_{i=1}^p\lambda_i\mu_jH^jGH^i\Delta t^{i+j}.
 \end{split}
\end{equation}
By recursively using $HG=GA_c$, for any $i,j$, we have 
\begin{equation}\label{th2:eq14}
H^iGA_c^j=GA_c^{i+j}=H^{i+j}G=H^jGA_c^i.
\end{equation}
According to (\ref{th2:eq13}) and (\ref{th2:eq14}), we have that (\ref{th2:eq12}) is true, i.e., the second condition (\ref{eq89}) is satisfied. 

For the third condition in (\ref{eq89}) we have
\begin{equation*}
\begin{split}
\tilde{H}(\Delta t)b=r(H\Delta t)b&=\sum_{i=0}^{\infty}\frac{r^{(i)}(-\rho)}{i!}(H\Delta t+\rho I)^ib\\
&=\sum_{i=0}^{\infty}\frac{r^{(i)}(-\rho)}{i!}(H\Delta t+\rho I)^{i-1}(H\Delta t+\rho I)b\\
&\leq\sum_{i=0}^{\infty}\frac{r^{(i)}(-\rho)}{i!}(H\Delta t+\rho I)^{i-1}\rho b\leq\sum_{i=0}^{\infty}\frac{r^{(i)}(-\rho)}{i!}\rho^ib=r(0)b=b.
\end{split}
\end{equation*}
Thus, the third condition in (\ref{eq89}) is also satisfied. The proof is complete. 
\end{proof}

The assumption $r(0)=1$ in Theorem \ref{thm31} is a fundamental condition for most discretization methods. This is since the steplength $\Delta t=0$, yielding  that the coefficient matrix of the discrete system is the identity matrix. 

\subsubsection{Compute Steplength Threshold}
The steplength threshold given in Theorem \ref{thm31} is related to $\rho$ and $\gamma.$ Recall that $\gamma$ is given in Lemma \ref{lemgam}, thus we only consider the computation of $\rho.$ 

Since $r(t)$ is a rational function,  all of its derivatives $r^{(i)}(t)$ have the same format, i.e.,  they are represented as quotients of two polynomial functions. Now recall that the radius of absolute monotonicity $\rho$ is defined as $r^{(i)}(t)\geq0$ for $t\in[-\rho,0].$ This requires that the polynomial function in the numerator of $r^{(i)}(t)$ is nonnegative for $t\in[-\rho,0]$.  Thus, a valid $\rho$ is the negative of the first negative real zero of this polynomial function.  Then an algorithm similar to the one presented in Section  \ref{sec:ty2} can be proposed to numerically compute $\rho.$ We are not repressing the algorithm here due to the space consideration.

\subsection{Parameter of Steplength Threshold}
According to Theorem \ref{thm123} and Theorem  \ref{thm31}, we have that the parameter $\gamma$ plays an important role to derive a large valid steplength threshold. In this section, we consider the effect of $\gamma$ to the steplength threshold. 

\subsubsection{Best Parameter}
Let us first consider the case for Taylor approximation type discretization methods. By simple modification, we have that $f_i(\Delta t)$ defined in (\ref{lemma453:eq3}) can be written as 
\begin{equation}\label{eq:291}
f_i(\Delta t)=\Delta t^i\sum_{k=i}^p\frac{(-1)^{k-i}}{k!}\binom{k}{i}(\gamma\Delta t)^{k-i}, \text{ for } i=0,1,...,p,
\end{equation}
which means that smaller $\gamma$ will yield larger steplength threshold for  Taylor type discretization method given as in (\ref{eq13}). Similarly, according to Theorem \ref{thm31}, we also have that  smaller $\gamma$ will yield larger steplength threshold for the rational function type discretization methods (\ref{eq:rat}).  Thus we prefer the smallest possible $\gamma$, which in fact can be computed by solving the following  optimization problem
\begin{equation}\label{opt12}
\min \{\gamma\,|\, H+\gamma I\geq0, ~HG=GA_c, \text{ and } Hb\leq0\}.
\end{equation}
In  optimization problem (\ref{opt12}),  the variables are $H$ and $\gamma$, while $G,A_c$ and $b$ are known, thus problem (\ref{opt12}) is a linear optimization problem,  which can be easily solved by  existing optimization algorithms, e.g., simplex methods \cite{bert} or interior point methods \cite{roos}. In particular, if there exists an $H\geq0$ such that $HG=GA_c$ and $Hb\leq0,$ then the optimal solution, denoted by $\gamma^*,$ of (\ref{opt12}) is nonpositive. In this case, according to (\ref{eq:291}), we have $f_i(\Delta t)\geq0$ for all $\Delta t\geq0$. Then according to 
the proof of Lemma \ref{leminvar}, we have that the steplength threshold for invariance preserving for  Taylor approximation type discretization methods (\ref{eq13}) on polyhedron $\mathcal{P}$ is infinity. Similarly, if $\gamma^*\leq0$, according to Theorem \ref{thm31}, we have that the steplength threshold for invariance preserving for  rational function type discretization methods (\ref{eq:rat}) on polyhedron $\mathcal{P}$ is also infinity. Thus, we have the following theorem. 

\begin{theorem}\label{lem:zero}
If the optimal solution of (\ref{opt12}) is nonpositive, then the steplength threshold for invariance preserving on the polyhedron $\mathcal{P}$ is infinity for  Taylor approximation type discretization methods (\ref{eq13}) and  rational function type discretization methods (\ref{eq:rat}).
\end{theorem}

One should note that the steplength thresholds given in Theorem \ref{thm123} and Theorem  \ref{thm31} may not be the largest steplength thresholds. For example, for the Taylor approximation type discretization methods, we aim to find the first positive zeros of finitely many polynomial functions. In fact, the first positive zeros may not be the best in some cases. For example, if  the function is given as  $f(\Delta t)=(\Delta t-1)^2(\Delta t-2)^2$, then its first positive zero is 1. Then, by our methods, we have $\tau=1$. However, it is clear that $f(\Delta t)\geq0$ for any $\Delta t\geq0.$ Thus, in this case, we have $\tau=\infty.$ 

If the first zero, $\Delta t^*$, of a function is a local minimum of this function, i.e., $f'(\Delta t^*)=0,$ then the first zero should not be used for computing the steplength threshold.  This is since the function is tangent to the $x$ axis at the first zero. To verify if a zero is a local minimum, one can check the first order and second order directives $f'(\Delta t^*)$ and $f''(\Delta t^*)$. If $f(\Delta t^*)=0$ and $f'(\Delta t^*)<0$, then we can say that $\Delta t^*$ is not a local minimum, and thus it is a valid positive zero. If $f(\Delta t^*)=0,f'(\Delta t^*)=0, $ and $f''(\Delta t^*)>0$, we can say that $\Delta t^*$ is a local minimum. Then we have to make $\Delta t$ to be larger, and use an algorithm similar to the one presented in Section \ref{sec:ty2} to find the next zero of $f(\Delta t)$.

\subsubsection{Relation to the Forward Euler Method}
The following lemma presents the relationship between $\gamma$ that satisfies the constraints in (\ref{opt12}) and the operator $I+\gamma^{-1}A_c$ on $\mathcal{P}.$ Recall that $I+\Delta t A_c$ is the coefficient matrix of the discrete system by using the forward Euler method.

\begin{lemma}\label{lem:reuler}
The conditions $H+\gamma I\geq0, HG=GA_c$, and $Hb\leq0$ are satisfied if and only if $(I+{\gamma}^{-1}A_c)\mathcal{P}\subseteq\mathcal{P}$.
\end{lemma}
\begin{proof}
$``\Rightarrow"$ For $x\in \mathcal{P}$, i.e., $Gx\leq b$, we have  
\begin{equation*}
\begin{split}
G(I+{\gamma}^{-1}A_c)x&=Gx+\gamma^{-1}GA_cx\\
&=Gx+\gamma^{-1}HGx  ~~~~~~\leftarrow \text{since } HG=GA_c\\
&=\gamma^{-1}(H+\gamma I)Gx\\
&\leq \gamma^{-1}(H+\gamma I)b~~~~~~~\leftarrow \text{since } Gx\leq b \text{ and } H+\gamma I\geq0\\
&=b+\gamma^{-1}Hb\leq b~~~~~~\leftarrow \text{since } Hb\leq0.\\ 
\end{split}
\end{equation*}
Thus we have $(I+\gamma^{-1}A_c)x\in \mathcal{P}$, i.e., $(I+\gamma^{-1}A_c)\mathcal{P}\subseteq \mathcal{P}$.

$``\Leftarrow"$ We note that $(I+\gamma^{-1}A_c)\mathcal{P}\subseteq \mathcal{P}$ means that $\mathcal{P}$ is an invariant set for the following discrete system:
\begin{equation*}
x_{k+1}=(I+\gamma^{-1}A_c)x_k.
\end{equation*}
Then according to Theorem \ref{poly}, we have that there exists an $\tilde{H}\in \mathbb{R}^{m\times m}$, such that $\tilde{H}\geq0, \tilde{H}G=G(I+\gamma^{-1}A_c),$ and $\tilde{H}b\leq b.$ Let $\hat{H}=\gamma \tilde{H},$ and then we have 
\begin{equation*}
\hat{H}\geq0, ~\hat{H}G=G(\gamma I+A_c), \text{ and } \hat{H}b\leq \gamma b,
\end{equation*}
i.e.,
\begin{equation*}
(\hat{H}-\gamma I)+\gamma I\geq0, ~(\hat{H}-\gamma I)G=GA_c, \text{ and } (\hat{H}-\gamma I)b\leq 0.
\end{equation*}
Thus replacing $\hat{H}-\gamma I$ by $H$, the proof is complete. 
\end{proof}

We highlight that  the forward Euler method is  used to analyze invariance in  continuous dynamical systems  in \cite{Blanchini3, blan5}.   In \cite{Blanchini3}, the largest domain of attraction of a continuous dynamical system is approximated with arbitrarily precision by using a polyhedral domain of attraction of a discrete dynamical system. This discrete dynamical system is  obtained by the forward Euler method and referred to as Euler approximating system in \cite{Blanchini3}. The value of $\gamma^{-1}$ in Lemma \ref{lem:reuler} can be considered as the step size of the forward Euler method for preserving the invariance of polyhedral $\mathcal{P}$, and the value of $\gamma$ is easily quantified.  The existence of a step size for preserving the contractivity of a set is also presented in \cite{Blanchini3} for the forward Euler method. 
A similar result to Lemma \ref{lem:reuler} is presented in \cite{blan5}, which is an extension of \cite{dorea2}, for (A,B)-invariance condition.  The forward Euler method is also applied to build the connection between continuous  and discrete dynamical systems. The value of the step size of the forward Euler method in \cite{blan5} for (A,B)-invariance condition is computed in a similar way to the one given as in  Lemma \ref{lem:reuler}.

\subsection{Forward Euler Method}

As illustration, we consider the simplest discretization method, the forward Euler method, in this section. For simplicity, a polytope, i.e., a bounded polyhedron, is chosen as the invariant set for the forward Euler method. A polytope can be defined in terms of  convex combination of its vetices, i.e., 
\begin{equation}\label{eq:tope}
\mathcal{P}=\text{conv}\{x^1,x^2,...,x^\ell\}=\Big\{x\,|\,x=\sum_{i=1}^\ell \lambda_ix^i, ~~\sum_{i=1}^\ell \lambda_i=1, ~~\lambda_i\geq0\Big\},
\end{equation}
where $\{x^i\}$ are the vertices of $\mathcal{P}$. A sufficient and necessary condition under which a polytope is an invariant set for the continuous system is presented below. 

\begin{lemma}\label{lem4}
\emph{\cite{song1}} The polytope $\mathcal{P}$ defined as in (\ref{eq:tope}) is an invariant set for the continuous system (\ref{dyna1}) if and only if
$A_cx^i\in\mathcal{T_P}(x^i), \text{ for } i=1,2,...,\ell,$
where $\mathcal{T_P}(x^i)$ is the tangent cone\footnote{
 The tangent cone  of a set
$\mathcal{S}$ at $x$, denoted by $\mathcal{T}_\mathcal{S}(x)$, is given as
$\mathcal{T}_\mathcal{S}(x)=\{y\in
\mathbb{R}^n\;|\;\underset{t\rightarrow0_+}{\lim\inf}\frac{{\text{dist}}(x+ty,\mathcal{S})}{t}=0\},
$
where   $\text{dist}(x,\mathcal{S})=\inf_{s\in\mathcal{S}}\|x-s\|.$} at $x^i$, which can be given 
\begin{equation}\label{eq5}
\mathcal{T_P}(x^i)=\{y\,|\,y=\sum_{j\neq i}\gamma_j(x^j-x^i),~\gamma_j\geq0\}.
\end{equation}
\end{lemma}

\begin{corollary}\label{cor:bound}
The polyhedron $\mathcal{P}$ defined as in (\ref{eq:tope}) is an invariant set for the continuous system (\ref{dyna1}) if and only if
there exist $\gamma_j^{(i)}\geq0, j=1,2,...,\ell$, such that
\begin{equation}\label{eq6}
A_cx^i=\sum_{j\neq i}\gamma_j^{(i)}(x^j-x^i), \text{ for all } i=1,2,...,\ell.
\end{equation}
Let $\epsilon^i=\big(\sum_{j\neq i}\gamma_j^{(i)}\big)^{-1}$ for $i=1,2,...,\ell,$ then 
\begin{equation}
x^i+\Delta tA_cx^i\in \mathcal{P} \text{ for any } \Delta t\in[0,\epsilon^i].
\end{equation}
\end{corollary}

\begin{proof}
According to Lemma \ref{lem4} and equation (\ref{eq5}),  equation (\ref{eq6}) is immediate. According to (\ref{eq6}) and $\epsilon^i\sum_{j\neq i}w_j^{(i)}=1$, we have
\begin{equation}\label{eq90}
\epsilon^iA_cx^i=\sum_{j\neq i}\epsilon^i\gamma_j^{(i)}(x^j-x^i)
=\sum_{j\neq i}\epsilon^i\gamma_j^{(i)}x^j-\sum_{j\neq i}\epsilon^i\gamma_j^{(i)}x^i
=\sum_{j\neq i}\epsilon^i\gamma_j^{(i)}x^j-x^i.
\end{equation}
 According to (\ref{eq90}), we have
$
x^i+\epsilon^iA_cx^i=\sum_{j\neq i}\epsilon^i\gamma_j^{(i)}x^j,
$ which is a convex combination of $\{x^j\}$, thus $x^i+\epsilon^iA_cx^i\in\mathcal{P}.$ For any  $\Delta t\in[0,\epsilon^i]$, by the convexity of $\mathcal{P},$ we have 
\begin{equation*}
x^i+\Delta tA_cx^i=\frac{\Delta t}{\epsilon^i}(x_i+\epsilon^iA_cx^i)+\frac{\epsilon^i-\Delta t}{\epsilon^i}x^i\in \mathcal{P},
\end{equation*}
which completes the proof. 
\end{proof}

We now consider the calculation of $\epsilon^i$, where $\epsilon^i$ is defined as in Corollary \ref{cor:bound}. By the formula of $\epsilon^i$, we need to compute $\gamma_j^{(i)}, j=1,2,...,\ell$, such that (\ref{eq6}) is satisfied. In fact, this can be achieved by solving the following optimization problem:
\begin{equation}\label{opt87}
  \min \Big\{\sum_{j\neq i}\gamma_j^{(i)}\,|\, \sum_{j\neq i}\gamma_j^{(i)}(x^j-x^i)=A_cx^i,~ \gamma_j^{(i)}\geq0.\Big\}
\end{equation}
Since $x^1,x^2,...,x^k$, and $A_c$ are known,  optimization problem (\ref{opt87}) is a linear optimization problem. One may obtain different values of $\hat\gamma_j^{(i)},  j=1,2,...,\ell,$ by choosing other objective functions in (\ref{opt87}). The advantage by using the current objective function in  (\ref{opt87}) is that this optimization problem yields the largest $\epsilon^i$ that satisfies (\ref{eq6}). This is since the  objective function in (\ref{opt87}) is  $(\epsilon^i)^{-1}.$
Thus, the value of  $\epsilon^i$ obtained by solving the optimization problem (\ref{opt87}) is the largest possible value of $\epsilon^i$. 

An alternative  is presented by the following discussion. Equation (\ref{eq5}) implies that $Ax^i$ is a feasible direction, i.e.,
$x^i+\tau^iA_cx^i\in \mathcal{P},$  for sufficiently small $ \tau^i>0.$
Then we can formulate the following linear optimization problem:
\begin{equation}\label{opt31}
  \max \Big\{\tau^i\,|\,\sum_{j=1}^\ell u_j^{(i)}x^j=x^i+\tau^iA_cx^i,~ \sum_{j=1}^\ell u_j^{(i)}=1,~ u_j^{(i)}\geq0\Big\}.
\end{equation}

Optimization problems (\ref{opt87}) and (\ref{opt31}) are equivalent problems, i.e., we claim that  $\tau^i$ is equal to $\epsilon^i$. Observing that $\sum_{j=1}^n\beta_j^{(i)}=1$ for  the first constraint in (\ref{opt31}), we have
\begin{equation}
\tau^iA_cx^i=\sum_{j=1}^\ell u_j^{(i)}x^j-\sum_{j=1}^\ell u_j^{(i)}x^i
          =\sum_{j=1}^\ell \tau^i\frac{u_j^{(i)}}{\tau^i}x^j-\sum_{j=1}^\ell \tau^i\frac{u_j^{(i)}}{\tau^i}x^i
          =\tau^i\sum_{j=1}^\ell \frac{u_j^{(i)}}{\tau^i}(x^j-x^i),
\end{equation}
i.e., $A_cx^i=\sum_{j\neq i} \frac{u_j^{(i)}}{\tau^i}(x^j-x^i).$ This,  by letting $\frac{u_j^{(i)}}{\tau^i}=\gamma_j^{(i)}$ gives the first constraint in (\ref{opt87}). 

According to the argument for $\epsilon^i$ above, we have the following theorem. 

\begin{theorem}\label{thm43}
Assume that the polytope $\mathcal{P}$ defined as in (\ref{eq:tope}) is an invariant set for the continuous system (\ref{dyna1}), and the forward Euler method is applied to (\ref{dyna1}). Then, $\tau=\min_{i=1,2,...,\ell}\{\epsilon^i\}$, where $\epsilon^i$ is defined as in Corollary \ref{cor:bound}, is the largest steplength threshold $\tau>0$ for invariance preserving of the forward Euler method on $\mathcal{P}$. 
\end{theorem}

\begin{proof}
For any $x\in\mathcal{P}$,  and $\Delta t\in[0,\tau],$ we have $x+\Delta tA_cx=\sum_{i=1}^{\ell}\lambda_i(x^i+\Delta tA_cx^i)$. According to Corollary \ref{cor:bound} and $0\leq\Delta t\leq \tau\leq \epsilon^i$, we have $x^i+\Delta tA_cx^i\in \mathcal{P}$. Thus we have $x+\Delta tA_cx\in \mathcal{P}.$ The proof is complete. 
\end{proof}

\section{Conclusions}\label{sec:con}
Many real world problems  are studied by developing dynamical system models.  In practice, continuous systems are usually solved by using discretization methods. In this paper, we consider  invariance preserving steplength thresholds on polyhedron, when the discrete system is obtained by using special classes of discretization methods. We particularly study three classes of discretization methods, which are: Taylor approximation type, rational function type, and the forward Euler method. 

For the first class of discretization methods, we show that a valid steplength threshold can be obtained by finding the first positive zeros of a finite number of polynomial functions. We also present a simple and efficient algorithm to numerically compute these positive zeros.  For the second class of discretization methods,  a valid  steplength threshold for invariance preserving is presented. This steplength threshold depends on the radius of absolute monotonicity, and can be computed by analogous method as in the first case. 
For the forward  Euler method we prove that the largest steplength threshold can be obtained by solving a finite number of linear optimization problems.

\section*{Acknowledgments} This research is supported by a Start-up grant of Lehigh University and by
TAMOP-4.2.2.A-11/1KONV-2012-0012: Basic research for the development of
hybrid and electric vehicles. The TAMOP Project is supported by the European Union
and co-financed by the European Regional Development Fund.


\medskip
Received xxxx 20xx; revised xxxx 20xx.
\medskip

\end{document}